%% file: paperV5.tex
\documentclass{article}

\input{command}

\usepackage{authblk}

\usepackage[style=numeric,backend=biber,maxnames=999]{biblatex}
\begin{document}
	\title{On the involution preserving properties   of semi-discrete Active Flux methods on simplex.}
\date{}
 \author[1]{R. Abgrall} 
 \author[2]{Y. Liu}
 \author[3]{M. Ricchiuto} 
 
 \affil[1]{Institute of Mathematics, University of Z\"urich, Switzerland}
 \affil[2]{University of Trento, Italy and University of Science and Technology, Shenzhen, China}
 \affil[3]{INRIA, Univ. Bordeaux, CNRS, Bordeaux INP, IMB, UMR 5251, 200 Avenue de la Vieille Tour, 33405 Talence cedex, France }
\date{\today}
\maketitle
\begin{abstract}
We study the involution-preserving properties of the Active Flux method on unstructured triangular meshes. Two model problems are considered: the two-dimensional linear acoustic system, which satisfies a curl involution on the velocity field, and the two-dimensional Maxwell equations with an in-plane magnetic field, which satisfy a divergence involution. Both systems can be written in conservative form. The question we address is how a single discretization can satisfy the conservative form and the involution at the same time. Within a discrete Active Flux framework on simplices, we derive explicit formulas for a discrete curl operator that is exactly preserved by the scheme for acoustics, and for a discrete divergence operator preserved for the Maxwell system. The analysis applies to general simplex meshes and to any polynomial degree, including extensions of the method beyond third order. Numerical experiments with the third-order scheme confirm the theoretical results.

\paragraph{Highlights}
\begin{itemize}
\item Active Flux on unstructured triangular meshes preserves discrete involutions,
\item Explicit discrete curl operator preserved by the scheme for linear acoustics,
\item Explicit discrete divergence operator preserved for 2D Maxwell equations,
\item Results hold on general simplex meshes,
\item Third-order numerical experiments confirm the theoretical results.
\end{itemize}
\end{abstract}

\section{Introduction}
We are interested in how to discretize    hyperbolic systems endowed with  involution. We focus on
the example of the two-dimensional  linear  acoustic system with a curl involution, but the same ideas can be easily extended to three space dimensions and to   the Maxwell equations with solenoidal magnetic fields. 

We start by introducing a few notations. Real numbers are denoted by non bold characters, like $a$. Vectors are denoted by bold characters, such as $\bba$, and matrices by bold capital letters, such as $\bbA$. If $\varphi$ is real valued and regular enough, its gradient will be denoted by $\nabla \varphi$ or $\grad \varphi$. In two space dimensions, we define the curl of the scalar $\varphi$ as  the vector $$\curl\varphi=\big (\dpar{\varphi}{y},-\dpar{\varphi}{x}\big ).$$ Conversely, if $\bm{\varphi}=(\varphi_1,\varphi_2)$ is vector valued, its rotational  is  the scalar $$\rot{\bm \varphi}=\dpar{\varphi_2}{x}
 -\dpar{\varphi_1}{y}.$$ Finally, its divergence is $$\div\bm{\varphi}=\dpar{\varphi_1}{x} +\dpar{\varphi_2}{y}.$$
 We have the relation $\div\curl=0$, $\rot\grad=0$ and the related exact sequences of the de Rham complex.

\subsection{PDE systems and involutions}

We consider here two model systems.
The acoustic problem writes (for unit density and unit speed of sound)
\begin{equation}\label{eq:1}
\begin{array}{ll}
\dpar{\bbu}{t}+\nabla p=0\\[10pt]
\dpar{p}{t}+  \text{ div }\bbu=0
\end{array}\qquad \forall\mathbf{x}\in\Omega\subset\mathbb{R}^2\;,\forall\,t \ge 0
\end{equation}
where in this work $\bbu=(u_1,u_2)\in \R^2$, and $p$ is  a scalar. The system \eqref{eq:1} is supplemented by initial conditions on the velocity and the pressure:
$$\bbu(\bbx,0)=\bbu_0(\bbx)\qquad p(\bbx,0)=p_0(\bbx).$$
As a consequence of the identity $\rot\grad=0$,
the system \eqref{eq:1} also satisfies, for any closed curve smooth enough curve $\Gamma$, the  curl involution (in integral form)
\begin{equation}\label{stokes}\int_\Gamma\bbu\cdot d\btau=\int_\Gamma\bbu_0\cdot d\btau.\end{equation}

\bigskip

We also consider the two dimensional system of Maxwell equations reading  (for unit speed of light)
\begin{equation}\label{eq:max1}
\begin{array}{ll}
\dpar{\mathbf{B}}{t}+ \curl E =0\\[10pt]
\dpar{E}{t} - \text{ rot }\mathbf{B} =0
\end{array} \qquad \forall\mathbf{x}\in\Omega\subset\mathbb{R}^2\;,\forall\,t \ge 0
\end{equation}
where $E$ is the (scalar) vertical component of the electric field, while  $\bbB=(B_1,B_2,0)^T$
is the magnetic field, assumed to lay in the horizontal plane.  As for acoustics, system \eqref{eq:max1} is also supplemented by initial conditions 
$$
\mathbf{B}(\mathbf{x},0)=\mathbf{B}_0(\mathbf{x})\quad
E(\mathbf{x},0)=E_0(\mathbf{x}).
$$
As a consequence of the identity $\text{div}\,\curl =0$, the system \eqref{eq:max1} satisfies, for any closed smooth enough curve $\Gamma$, the div involution (in integral form)
\begin{equation}\label{gauss}\int_\Gamma\mathbf{B}\cdot d\mathbf{n}=\int_\Gamma\mathbf{B}_0\cdot d\mathbf{n}.\end{equation}

Note that both systems above can be written in conservative form as
\begin{equation}\label{eq:con}
\dpar{\mathbf{U}}{t} +\nabla\cdot\mathbf{f} =0
\end{equation}
where   for acoustics we have
$$
 \mathbf{U}=\left(\begin{array}{c} \bbu\\ p\end{array}\right) \;,\quad
 \mathbf{f} = \left(\begin{array}{cc} p & 0\\0& p\\u_1&u_2\end{array}\right),
$$
while 
$$
 \mathbf{U}=\left(\begin{array}{c} \mathbf{B}\\ E\end{array}\right) \;,\quad
 \mathbf{f} = \left(\begin{array}{cc} 0 & E\\-E& 0\\-B_2&B_1\end{array}\right),
$$
for the Maxwell system.
This work deals with the design of methods   that satisfy simultaneously  
the conservative form  \eqref{eq:con}, and the corresponding
involution \eqref{stokes} or   \eqref{gauss}, possibly with 
   an  explicit definition of the
   appropriate 
discrete curl and   div operators. 
Of course, for a linear problem, this is not very important, but we want to investigate on simple problems how to manage together these two constraints.

\subsection{A (partial) state of the art}

The numerical approximation of PDE problems embedding involutions has seen over the years
many contributions. The earliest are  related to  so-called mimetic finite difference operators see e.g.  
\cite{KYee} and \cite{HYMAN199781,MARGOLIN2000365,LIPNIKOV20141163}.  Some of the notions behind these methods naturally generalize to polygonal meshes.  Some of the ideas underpinning these finite difference 
schemes have been  embedded over the years in constraint preserving/compatible finite volume methods \cite{BALSARA1999270,BALSARA2017604,balsara_camc23,DelRio25}. Of course, another natural setting
to obtain discretizations compatible with differential constraints is provided by 
compatible and mixed finite element spaces \cite{arnold,NedelecI,NedelecII}. 
The issue of the curl preservation has become more and more central 
for the design of robust Lagrangian methods 
\cite{SIDILKOVER2025113510,ovadia,MAIRE20092391,sambasivan2013finite}.
In this setting, the benefits of  
genuinely multidimensional methods has become more and 
more fundamental \cite{BALSARA2014163,BALSARA2015269,Barsukow,GALLICE2022111493,AbgrallM3AS,GABURRO2026107030,ABGRALL2025114340}.
The development of such methods has benefited in several respects of seminal ideas in
early works by Morton and Roe  on vorticity preserving schemes\cite{MortonRoe}, and P.L. Roe and H. Deconinck
on multidimensional upwinding  \cite{DECONINCK1993215}.
Other techniques have also surfaced, based on appropriate enrichment of the approximation spaces
for finite volumes and discontinuous Galerkin
 \cite{JUNG2024113252,VPerrier}, 
or on formulations based on flux potentials \cite{jeltsch06,mishra09preprint}.
Curl preservation has also been shown to be tightly related to the reduction of mesh imprinting \cite{lung14}, and to the existence of a rich stationary kernel (stationary preservation) \cite{Barsukow2019}. This has also led to other methods based on
 multidimensional quadrature techniques \cite{barsukow2024structure,barsukow2026stationarity,barsukow2026genuinely,ziggaf2026arbitrary}, which provide one possible  generalization of   the method by Morton and Roe.\\

One approach which fits perfectly in this panorama is the Active Flux method. 
Previous literature has shown that Active Flux has stationary preserving features,
which lead to involution preservation in the linear case of acoustics.
These features have been shown   either experimentally or by Fourier analysis
\cite{Barsukow2023,Barsukow2026}.
 In the case of unstructured meshes, it seems that the methods is also friendly with respect to \eqref{stokes}, i.e. one can expect the existence of  a discrete version of \eqref{stokes}, a little bit like in   \cite{MortonRoe}. However the mere definition of this discrete version of the curl is not known.

\subsection{Contribution of this paper}

The purpose of this work is to study the involution preserving properties of active flux 
on unstructured meshes. In particular, we consider   the discrete active flux setting
described in \cite{Abgrall2023b}  and exhibit    discrete versions of \eqref{stokes} for acoustics, and \eqref{gauss} for the Maxwell system. 
Despite the unstructured grid setting, we are able to provide explicit formulas
for the discrete curl and div operators preserved by the method.  
The theory can be applied to simplex meshes, and to any degree of approximation.
Numerical results   for the third order scheme are provided    to  support the theoretical claim. 

The paper is organized as follows. We first discuss in Section \S2 the discrete polynomial approximation of the variables on triangles. Then, in Section  \S3 we present the Active Flux method for the linear acoustics system and Maxwell equations, and a generalization to orders higher than three. Section \S4 provides the main results of the paper, namely explicit 
discrete definitions of the curl and div which are  involutions  for the Active Flux in the case of the acoustics system and Maxwell equations respectively.
Numerical results to support the theory are provided in Section \S5. The paper is ended by remarks and open avenues to improve the results discussed.

\section{Approximation: triangular case}
The domain $\Omega$ is covered by a family of non overlapping polygons denoted by $K$. This paper focuses on the case in which
the polygons are triangles. Some comments  on the general case are provided along the text.
 We assume that the mesh is regular. The symbol $e$ will denote any edge of $K$, and $\btau_e$ (or $\btau$ if there is no ambiguity), is the tangent unit vector at $e$. Similarly, $\bbn_e$ is the outward normal unit vector of the edge $e$ seen from $K$.
 In practice, we will choose one orientation of the edges of the triangulation, and will choose $\btau_e$ and $\bbn_e$  accordingly so that $(\btau_e,\bbn_e)$ is direct.
 
 
Let us start from the lowest order case: linear and quadratic.  
The unknown $\mathbf{U}$ is  approximated by means of globally continuous polynomials. The approximation involves collocated values 
    at every vertex of the mesh, and in the quadratic case, at the mid-points of the edges: these are the Lagrange points. In addition we take as independent degree of freedom the average in the element. Hence, denoting for a generic triangle $K$ with vertices $\{\bba_1,\bba_2,\bba_3\}$ the barycentric coordinates by $\lambda_1, \lambda_2,\lambda_3$, we consider 
 $\bar\varphi=60\lambda_1\lambda_2\lambda_3$. The indexation of the vertices respects the positive orientation of $K$.This function vanishes on the boundary of $K$ and
 $$\int_K\bar\varphi \; d\bbx=\vert K\vert.$$
 Taking $\{\psi_i\}_{i=1, N}$ the Lagrange polynomials (so $N=3$ and $\psi_i=\lambda_i$ or $N=6$), we define
 $$\varphi_i=\psi_i-\frac{\int_K\psi_i\; d\bbx}{\vert K\vert}\bar\varphi.$$
 These functions are still interpolant and have a zero mass in $K$. The unknown is thus approximated by
 \begin{equation}
 \label{approx:pression}
 \mathbf{U}=\sum_{i=1}^N \mathbf{U}(\bba_i)\varphi_i+\overline{ \mathbf{U}}_K\bar\varphi.
 \end{equation}

\subsection{Higher order extension}
The higher order extension is done using the notation introduced in
\cite{HO_AB_WB} for the general polygonal case, and detailed in 
 \cite{OEPampa} and \cite{pampa_dg} for the triangular case. Similar ideas are used in  \cite{WB_PC_CK_LL}  for Cartesian meshes. 
 
 We want to have a globally continuous approximation which reduces to polynomials of degree $k>2$ on the boundary of each element.  We also want that the approximation space contains $\P^k$ for the accuracy.  We denote the vertices by $\bba_i$, $i=1,2,3$ and $\lambda_i$ the barycentric coordinates of $K$. The DoFs are these vertices and the other $k-1$ Gauss--Lobatto points on the three edges. To these we add  $k(k-1)/2$  moments \cite{BP_Pampa_VEM}. All of these DoFs lead to an  approximation space with local dimension $N_{\dof}=3+3(k-1)+k(k-1)/2=3k+k(k-1)/2$. 

The dimension of $\P^k$ is $(k+1)(k+2)/2$. The number of Lagrange points on each edge, taking into account the vertices as Lagrange points, is $3k$. We have
$$3k+k(k-1)/2-3k=\frac{k(k-1)}{2}={\text{dim}\P^{k-2}}.$$
These equalities can be rephrased as follows. We set $B_\bmu=c_\bmu \lambda_1^{\mu_1}\lambda_2^{\mu_2}\lambda_3^{\mu_3}$ the B\'ezier polynomials. For the sake of simplicity, we will write
$$\lambda_1^{\mu_1}\lambda_2^{\mu_2}\lambda_3^{\mu_3}=\blambda^\bmu,$$
with   $\bmu$ the multi-index $\bmu=(\mu_1,\mu_2,\mu_3)$. We set $\vert \bmu\vert=\mu_1+\mu_2+\mu_3$.  The coefficients $c_\bmu$ are binomial coefficients so that the $B_\bmu$s can be seen as the monomials             in the development of $(\lambda_1+\lambda_2+\lambda_3)^k$.  We define the two vector spaces:
$$V_1=\text{span}\{ B_\bmu, \vert \bmu\vert=k \text{ and } \mu_1\mu_2\mu_3=0\}$$
and
$$V_2=\text{span}\{ B_\bmu, \vert \bmu\vert=k \text{ and } \mu_1\mu_2\mu_3>0\}.$$
Clearly, $\P^k\subset V_1\bigoplus V_2$. The elements of $V_2$ vanish on the boundary of $K$. For the sake of clarity, we will denote the generator of $V_2$ with a specific notation
$$\psi_{\bmu}=\lambda_1\lambda_2\lambda_3 B_{\bmu}$$ 
with
\begin{itemize}
\item if $1\leq k\leq 2$, $B_{\bmu}=1$,
\item and for $k\geq 3$, $B_{\bmu}=\blambda^{\bmu}$ with 
$$\vert \bmu\vert=k-2.$$
\end{itemize}
It is clear that $\text{dim }V_1=3k$ and $\text{dim }V_2=\text{ dim }\P^{k-2}$. We define
$$\mathcal{I}_k=\left \{ \begin{array}{ll}
\{\bmu, \vert \bmu\vert=k-2\}&\text{ for }k>2\\
\{(0,0,0)\}  &\text{ for }k\leq 2.
\end{array}
\right .$$
The cardinal of $\mathcal{I}_k$ is {$\text{ dim }\P^{k-2}$}. 
We will often write $b=\lambda_1\lambda_2\lambda_3$.

\bigskip

\noindent We introduce the following polynomial spaces:
\begin{itemize}
\item For $k\leq 2$, $V=\P^k\bigoplus b\R$,
\item and for $k>2$, $V=\P^k \bigoplus b\P^{k-2}$
\end{itemize}
For the approximation of  a generic variable $p$, we introduce the following degrees of freedom:\begin{subequations}\label{linear forms}
\begin{itemize}
\item For the Gauss-Lobatto points $\sigma$ on the boundary of $K$,
\begin{equation}\langle \delta_\sigma, p\rangle =p(\sigma).\label{linear form points}\end{equation}
The set of Gauss-Lobatto points is denoted by $\mathcal{G}$ and contains $3k$ points.
\item In the interior,
\begin{equation}\langle m_{\bmu}, p\rangle =\frac{1}{\vert K\vert}\int_K \blambda^{\bmu}(\bbx)p(\bbx)\; \mathrm d\bbx\label{linear forms moment}\end{equation}
for $\bmu\in \mathcal{I}_k$.
\end{itemize}\end{subequations}
They are linear forms on $V$.
\begin{lemma}\label{lem:unisolv}
The linear forms \eqref{linear form points}-\eqref{linear forms moment} are uni-solvant on $V$.
\end{lemma}
\begin{proof}
See appendix \ref{app:lemma}.
\end{proof}

\begin{figure}
\begin{center}
\begin{tabular}{c}
\input{GL_tri.tikz}
\end{tabular}
\end{center}\label{fig:1}
\caption{Interpolation points for the triangular cubic case.}
\end{figure}
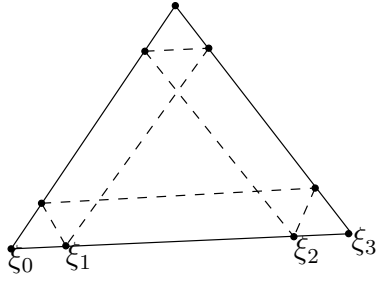

As shown in the appendix, 
for example for $k=3$  the polynomials are
\begin{itemize}
    \item Moment $\lambda_1$: $\varphi_{(1,0,0)}=\lambda_1\lambda_2\lambda_3(a_1\lambda_1+a_2\lambda_2+a_3\lambda_3)$.
    \item Moment $\lambda_2$: $\varphi_{(0,1,0)}=\lambda_1\lambda_2\lambda_3(a_3\lambda_1+a_1\lambda_2+a_2\lambda_3)$.
    \item Moment $\lambda_3$: $\varphi_{(0,0,1)}=\lambda_1\lambda_2\lambda_3(a_2\lambda_1+a_3\lambda_2+a_1\lambda_3)$.
    \end{itemize}
Note that 
$$
\varphi_{(1,0,0)}+\varphi_{(0,1,0)}
+\varphi_{0,0,1)}=
(a_1+a_2+a_3) b
$$
which means that the sum of the three corresponding DoF is a bubble representing the average.


 \section{Discretization of the PDEs}
 \subsection{Third order scheme}
 The average will evolve by:
 \begin{equation}
     \label{average}
     \vert K \vert \dfrac{d\overline{\mathbf{U}}_K}{dt}+\int_{\partial K}\bbf(\mathbf{U})\bbn \;d\gamma=0
 \end{equation}
 Since $\mathbf{U}$ is  
 known on the boundary (it is quadratic), one can exactly evaluate the integral by using the Simpson formula on each edge. This is what we do.

 For the point value update, we proceed as in \cite{Abgrall2023b}:
 \begin{equation}
     \label{point values}
     \dfrac{d\mathbf{U}_\bsigma}{dt}+\sum_{K, \bsigma\in K} \bomega_\bsigma^K
     \nabla\cdot\mathbf{f}_{\mid K}(\bsigma)
=0
 \end{equation}
 where
 \begin{itemize}
     \item $\mathbf{f}_{\mid K}$ is the flux 
     polynomial (in our case the polynomial approximation of the velocity and the pressure, resp. of the magnetic field and of the electric field)
     in $K$,
     \item $\bbn_\bsigma^K$ is the normal defined by
     \begin{itemize}
\item For $\bsigma=1, 2, 3$, this is the normal to the side opposite to the vertex $\bba_i$,
\item for $\bsigma=4, 5, 6$, if they are considered, this is the outward normal of the side on which the degree of freedom is sitting.
\end{itemize}
The normal also depends on $K$, so we name then $\bbn_\bsigma^K$.
     \item $\bomega_\bsigma^K$ is the matrix defined by
     $$\bomega_\bsigma^K= \bbN\bigg ( \big (\bbK\bbn_\bsigma^K\big )^++\varepsilon \Id\bigg )$$
     with $\bbK\bbn$ the matrix
     $$
     \bbK\bbn= \dfrac{\partial\mathbf{f}\cdot \mathbf{n}}{\partial\mathbf{U}},
     $$
     and with 
     $$\bbN^{-1}=\sum_{K, \bsigma\in K}  \bigg ( \big (\bbK\bbn_\bsigma^K\big )^++\varepsilon \Id\bigg )$$ for some small $\varepsilon\geq 0$, and with the usual definition 
     $$\bbK_\bbn^+=\frac{\bbK\bbn+\vert \bbK\bbn\vert }{2}$$
 \end{itemize}

In appendix \ref{app:1}  we show that
the matrix $\mathbf{N}_{\sigma}=\sum_{K, \sigma\in K}\big (\bbK^K_\bbn)^+ $ is always invertible for internal nodes for both acoustics and the Maxwell system. For this reason, in practice we set $\varepsilon=0$ everywhere except for the boundary points.
Typically, we take  for these nodes $\varepsilon=\vert K\vert/2$.

    \subsection{Higher order extensions}
    We follow the same technique. The moment are updated by
    \begin{equation*}
        \vert K\vert \dfrac{d m_\bmu}{dt}+\int_K \blambda^\bmu(\bbx)\text{div }\bbf\; d\bbx=0
    \end{equation*}
    that is
       \begin{equation}
        \vert K\vert \dfrac{d m_\bmu}{dt}-\int_K \nabla \blambda^\bmu(\bbx)\bbf\; d\bbx+\int_{\partial K} \blambda_\bmu \bbf\bbn \; d\gamma=0
    \end{equation}
    Since $\mathbf{U}$ is of degree $k$ on the boundary of $K$, the boundary integral is computable. Similarly, the surface integrable can be evaluated from a quadrature formula.

    The values of $\mathbf{U}$ at the Gauss-Lobatto points as  evaluated as    in the quadratic case \cite{Abgrall2023b}. 

\section{Discrete involutions}

\subsection{Acoustics and nodal curl involutions}
We consider a triangle $K$ and for the moment the case $k\le 2$. The averages in each element   are updated by \eqref{average}. In particular, 
in the case of acoustics we 
have 
$$\vert K \vert \dfrac{d\overline{\bbu}_K}{dt}+\int_{\partial K} p \bbn \,d\gamma=0.$$
for the velocity.

Consider now $\varphi_\bsigma$, the $\P_1$ basis function associated to a \textit{vertex} $\bsigma$, and  the nodal area 
  $|C_\bsigma|=\sum_{K, \sigma\in K}\vert K\vert$. We consider the following quantity 
\begin{equation}
    \label{rot discret}
\ROT\bbu_\sigma:=\sum_{K, \sigma\in K} \dfrac{|K|}{|C_\bsigma|}\left(\dpar{\varphi_\sigma}{x}\bar u-\dpar{\varphi_\sigma}{x} \bar v\right).
\end{equation}
Since the gradient of $\varphi_\bsigma$ is constant in $K$, we see that
$$\dpar{\varphi_\sigma}{x}\bar u-\dpar{\varphi_\sigma}{x} \bar v=\dfrac{1}{|K|}
\int_K \nabla\varphi_\bsigma^\bot \cdot (\bar u_K,\bar u_K)\; d\bbx$$ that we can also rewrite as 
$$\dpar{\varphi_\sigma}{x}\bar u-\dpar{\varphi_\sigma}{x} \bar v=\dfrac{1}{|K|}
\int_K \nabla\varphi_\bsigma^\bot \cdot \bbu\; d\bbx.$$
Hence
$$\ROT\bbu_\bsigma=\frac{1}{\vert C_\bsigma\vert}\int_\Omega\nabla\varphi_\sigma^\bot \cdot \bbu \; d\bbx.$$
This can also be seen as a "weak" form of the curl:
$$\int_K \nabla\varphi_\bsigma^\bot \cdot \bbu\; d\bbx=\int_K\curl\big ( \varphi_\bsigma \bbu \big )d\bbx-\int_K \varphi_\bsigma \curl \bbu \; d\bbx=\int_{\partial K} \varphi_\sigma \bbu\cdot\btau\,d\gamma -\int_K \varphi_\bsigma \curl \bbu \; d\bbx.$$
Then we take the sum and get
$$\ROT\bbu_\bsigma=\dfrac{\int_K \varphi_\sigma \curl \bbu}{\int_K \varphi_\bsigma}$$ because $\varphi_\sigma$ vanishes on the boundary of the support of $\varphi_\bsigma$ and each edge starting form $\bsigma$ is counted twice with opposite sign: $\btau$ is the tangent of the boundary and is seen positively in each triangle $K$.

We can show the following result.
\begin{proposition}
\begin{equation}
    \label{curl}
\dfrac{d\ROT\bbu_\bsigma}{dt}=0.
\end{equation}
\end{proposition}
\begin{proof}
We have
$$\vert C_\bsigma\vert \dfrac{d\ROT\bbu_\bsigma}{dt}=\sum_{K\bsigma\in K}  \nabla\varphi_\bsigma^\bot \cdot \dfrac{d}{dt}\int_K  \bbu \,d\bbx=
-
\sum_{K\bsigma\in K}  \int_K \nabla\varphi_\bsigma^\bot \cdot \nabla p\; d\bbx
$$
because $\nabla\varphi_\bsigma$ is constant and 
$$\vert K\vert \dfrac{d\bar\bbu_K}{dt}=-\int_{\partial K} p\bbn\,d\gamma=-\int_K\nabla p\; d\bbx$$
because $\bbu$ is smooth in $K$.
Then we see that 
$$\int_K \nabla\varphi_\bsigma^\bot \cdot \nabla p\; d\bbx=\int_{K }\bigg (
\dpar{\varphi_\bsigma}{y}\dpar{p}{x}-\dpar{\varphi_\bsigma}{x}\dpar{p}{y}\bigg )\; d\bbx$$
and since
$$\dpar{\varphi_\bsigma}{y}\dpar{p}{x}-\dpar{\varphi_\bsigma}{x}\dpar{p}{y}=
\dpar{}{y}\bigg ( \varphi_\bsigma\dpar{p}{x}\bigg )-\dpar{}{x}\bigg ( \varphi_\bsigma\dpar{p}{y}\bigg )$$
because in $K$, we can apply Schwarz theorem on $p$ (\ldots), we get
using Stokes formula,
$$\int_K \nabla\varphi_\bsigma^\bot \cdot \nabla p\; d\bbx=\int_{\partial K} \nabla p\cdot \btau \varphi_\bsigma\,d\gamma.$$
Then using the same arguments as before (the tangential derivative of $p$ is counted twice with opposite signs, $\varphi_\bsigma$ is continuous accorss edges and vanishes on the boundary of its support), we see that \eqref{curl} is true.
\end{proof}

\begin{remark}\label{rem:simplex}
$~$
\begin{enumerate}
    \item 
We get one involution per vertex. The proof uses the fact that the gradient of a piecewise linear function is constant per element and thus since the characteristic function of the triangle are used to compute the first moment (the average), we can use this gradient as a test function. 
 For mid points we cannot do the same.
 \item Similarly, the proof does not work for polygons, whatever degree because we would need the basis function attached to the $\bsigma$, and the moments are not expressed with this.
 \end{enumerate}
 \end{remark}

 In the higher degree case, let us denote by  $\blambda^\bmu_{k-1}$ the moment polynomials for $\vert \bmu\vert =k-1$. 
 We can  see that the gradient of $\blambda^\bmu_{k-1}$   is of degree $k-2$: it can be written as a linear combination of the $\blambda^\bmu_{k-2}$ with $\vert\bmu\vert=k-2$. Since these polynomials   are used to define the moments, we can use $\nabla\blambda^\bmu$ as test function. 
 In particular  note that we can introduce the gradients   in the moments of  the approximation 
 of degree $k\ge 3$ of a variable $p$:
 \begin{equation}\label{eq:grad_mom1}
  \dfrac{1}{|K|}
\int_K  \nabla \blambda^\bmu_{k-1}(\bbx) p(\bbx)\,d\bbx=
\sum_{|\bmu|=k-2} \bba_{\bmu} 
\int_K \blambda^\bmu_{k-2}(\bbx) p(\bbx)\,d\bbx =
\sum_{|\bmu|=k-2} \bba_{\bmu} 
\langle m_\bmu, p
\rangle 
 \end{equation}
 for some  set of vector coefficients  $\bba_{\bmu} $,
 and similarly we have
 \begin{equation}\label{eq:grad_mom2}
  \dfrac{1}{|K|}
\int_K  (\nabla \blambda^\bmu_{k-1})^\bot(\bbx) p(\bbx)\,d\bbx=
\sum_{|\bmu|=k-2} \mathbf{b}_{\bmu} 
\int_K \blambda^\bmu_{k-2}(\bbx) p(\bbx)\,d\bbx =
\sum_{|\bmu|=k-2} \mathbf{b}_{\bmu} 
\langle m_\bmu, p
\rangle 
 \end{equation}
 for some  set of vector coefficients  $\mathbf{b}_{\bmu} $.
 
 Based on this observation we can prove the following  result.
 \begin{proposition}\label{prop:curlHO}
 Let $k\geq 3$ and $\bmu=(\mu_1,\mu_2,\mu_3)$ with $\vert \bmu\vert=\mu_1+\mu_2+\mu_3=k-1\geq 2$.
 We set $$B_\bmu=\frac{k-1}{\mu_1!\mu_2!\mu_3!}\blambda^\bmu_{k-1},$$ the B\'ezier polynomial on $K$
     Then the weak curls $\ROT\bbu_\bsigma$ defined by
     $$
     \int_{\Omega}\blambda_{k-2}^{\bmu}(\bbx)\ROT\bbu_\bsigma\,d\bbx=
     \int_\Omega  \nabla B_\bmu ^\bot \cdot \bbu\,d\bbx$$
     verify
      $$
      \dfrac{d}{dt} \ROT\bbu_\bsigma =0.
      $$
      when $\bbu$ is the degree $k\ge 3$ active flux
          approximation of velocity in the  acoustics  system.
 \end{proposition}
 \begin{proof}
     Applying twice  \eqref{eq:grad_mom2} we can write
     as  
     $$  
     \dfrac{d}{dt}\int_K \big (\nabla\blambda^\bmu_{k-1}\big )^\bot \cdot \bbu\,d\bbx=\sum\limits_{\vert \mu \vert =k-2} \mathbf{b}_{\bmu}\cdot 
      \dfrac{d}{dt}\int_K  \blambda^\bmu_{k-2} \bbu\,d\bbx=
     - \sum\limits_{\vert \mu \vert =k-2} \mathbf{b}_{\bmu} \cdot \int_K  \blambda^\bmu_{k-2} \nabla p\,d\bbx=
     -   \int_K \big (\nabla\blambda^\bmu_{k-1}\big )^\bot \cdot \nabla p\,d\bbx
     $$
having also used
$$
 \dfrac{d}{dt}\int_K  \blambda^\bmu_{k-2} \bbu\,d\bbx =
  - \int_K  \blambda^\bmu_{k-2}\nabla p\,d\bbx
$$
This allows to write
     $$
     \dfrac{d}{dt}\int_K \big (\nabla\blambda^\bmu_{k-1}\big )^\bot \cdot \bbu\,d\bbx=-\int_{\partial K} \blambda^\bmu_{k-1} \dpar{p}{\btau}\; d\gamma.
     $$
     Then
     \begin{equation*}
     \begin{split}\int_\Omega \blambda^\bmu_{k-2} \dfrac{d\ROT\bbu_\bsigma}{dt}&=\sum_{K\subset \text{supp}(B_\bmu)}\int_K \big (\nabla B_\bmu\big )^\bot\cdot \dfrac{d\bbu}{dt}\\
     &=-\sum_{K\subset \text{supp}(B_\bmu)}\int_K \big (\nabla B_\bmu\big )^\bot\cdot \nabla p\; d\bbx\\
     &=-\frac{k-1}{\mu_1!\mu_2!\mu_3!}\sum_{K\subset \text{supp}(B_\bmu)}\int_{\partial K}\blambda^\bmu_{k-1} \dpar{p}{\btau}\\
     &=0
     \end{split}
     \end{equation*}
     because on the internal edges, $\blambda^\bmu_{k-1}$ is continuous and $\dpar{p}{\btau}$ is counted twice with opposite sign, and on the boundary edges, $\blambda^\bmu_{k-1}=0$.
 \end{proof}

\begin{remark}\label{rem:3d}
The above results provides a justification of why AF works so well on triangle for acoustics: in the quadratic, we get $n_v$ involutions ($n_v=3$ being the number of vertices). In the case of $k>2$, we get $n_v+(k-2)n_{edges}= 3 + 3(k-2)$ involutions, where $n_{edge}=3$ is the number of edges.

    This result can be easily be extended to the 3D case on thetrahedra. The number of involutions will be $n_v+\max(k-2,0)n_{edge}+\max(k-3,0)n_{faces}$. Here $n_v=4$, and $n_{face}=4$ is the number of faces, while $n_{edge}=6$.
\end{remark}

\subsection{Maxwell equations and nodal div involutions}

We proceed in a similar manner,
and  start from  the case $k\le 2$. 
From \eqref{average}, the average magnetic field is evolved according to  
$$\vert K \vert \dfrac{d\overline{\bbB}_K}{dt}=-\int_{\partial K} E \btau\,d\gamma =- \int_{\partial K} \curl E\,d\bbx.$$
As before, consider $\varphi_\bsigma$, the $\mathbb{P}_1$ bases associated to the vertex $\bsigma$, and define the quantity
\begin{equation}
    \label{div discret}
\DIV\bbB_\sigma:=\sum_{K, \sigma\in K} \dfrac{|K|}{|C_\bsigma|} \bar\bbB_K \cdot \nabla\varphi_\bsigma.
\end{equation}
Using the fact that $\nabla\varphi_\bsigma$ is constant, we can see that  the above defines a   "weak" approximation of the divergence:
$$\int_K \nabla\varphi_\bsigma  \cdot \bbB\; d\bbx=\int_K\text{div}\big ( \varphi_\bsigma \bbB \big )d\bbx-\int_K \varphi_\bsigma \text{div} \bbB \; d\bbx=\int_{\partial K} \varphi_\sigma \bbB\cdot\bbn\,d\gamma -\int_K \varphi_\bsigma \text{div} \bbB \; d\bbx.$$
Then we take the sum and get
$$\DIV\bbB_\bsigma=\dfrac{\int_K \varphi_\sigma \text{div} \bbB\,d\bbx}{\int_K \varphi_\bsigma d\bbx}$$ because $\varphi_\sigma$ vanishes on the boundary of the support of $\varphi_\bsigma$ and each edge starting form $\bsigma$ is counted twice with opposite sign: $\bbn$ is the normal of the boundary and is seen positively in each triangle $K$.

We can prove the following.
\begin{proposition}
\begin{equation}
    \label{div}
\dfrac{d\DIV\bbB_\bsigma}{dt}=0.
\end{equation}
\end{proposition}
\begin{proof}
We have
$$\vert C_\bsigma\vert \dfrac{d\DIV\bbB_\bsigma}{dt}=\sum_{K\bsigma\in K}  \nabla\varphi_\bsigma \cdot \dfrac{d}{dt}\int_K  \bbB \,d\bbx=
-
\sum_{K\bsigma\in K}  \int_K \nabla\varphi_\bsigma \cdot \curl E\; d\bbx
$$
having used the evolution of $\bar\bbB_K$. 
Using Schwarz theorem for $E$, we see that 
$$\int_K \nabla\varphi_\bsigma  \cdot \curl E\; d\bbx=\int_{K }\bigg (
\dpar{\varphi_\bsigma}{x}\dpar{E}{y}-\dpar{\varphi_\bsigma}{y}\dpar{E}{x}\bigg )\; d\bbx
= 
\dpar{}{x}\bigg ( \varphi_\bsigma\dpar{E}{y}\bigg )-\dpar{}{y}\bigg ( \varphi_\bsigma\dpar{E}{x}\bigg ).
$$
Finally we obtain
$$\int_K \nabla\varphi_\bsigma  \cdot \curl E\; d\bbx=
\int_{\partial K} \nabla E\cdot \btau \varphi_\bsigma.$$
Then using the same arguments as before (the tangential derivative of $E$ is counted twice with opposite signs, $\varphi_\bsigma$ is continuous across edges and vanishes on the boundary of its support), we see that \eqref{div} is true.
\end{proof}

So we have a div involution for each vertex, with the same limitations already pointed out in Remark \ref{rem:simplex}.

For the high order case $k\ge 3$, we can proceed as before, and use the relation between moments $\lambda^\bmu$ of different degrees, and in particular \eqref{eq:grad_mom1}.
Based on this, we can show the following.

\begin{proposition}\label{prop:divHO}
 Let $k\geq 3$ and $\bmu=(\mu_1,\mu_2,\mu_3)$ with $\vert \bmu\vert=\mu_1+\mu_2+\mu_3=k-1\geq 2$.
 We set $$B_\bmu=\frac{k-1}{\mu_1!\mu_2!\mu_3!}\blambda^\bmu_{k-1},$$ the B\'ezier polynomial on $K$
     Then the weak divs $\DIV\bbB_\bsigma$ defined by
     $$
     \int_{\Omega}\blambda_{k-2}^{\bmu}(\bbx)\DIV\bbB_\bsigma\,d\bbx=
     \int_\Omega  \nabla B_\bmu \cdot \bbB\,d\bbx$$
     verify
      $$
      \dfrac{d}{dt} \DIV\bbB_\bsigma =0.
      $$
      when $\bbB$ is the degree $k\ge 3$ active flux
          approximation of Magnetic field in the  Maxwell  system.
 \end{proposition}
 \begin{proof}
     The proof is very similar to the one of \ref{prop:curlHO}:
  the repeated application  of \eqref{eq:grad_mom1}, and the evolution of the moments of the magnetic field   
   allow to show 
     $$
     \dfrac{d}{dt}\int_K \big (\nabla\blambda^\bmu_{k-1}\big )  \cdot \bbB\,d\bbx=-\int_{\partial K} \blambda^\bmu_{k-1} \dpar{E}{\btau}\; d\gamma.
     $$
<     Then, using the definition of the weak divergence moments, and the same continuity arguments of the tangent derivatives used before  we show that 
     \begin{equation*}
     \begin{split}\int_\Omega \blambda^\bmu_{k-2} \dfrac{d\DIV\bbB_\bsigma}{dt}
     &=-\frac{k-1}{\mu_1!\mu_2!\mu_3!}\sum_{K\subset \text{supp}(B_\bmu)}\int_{\partial K}\blambda^\bmu_{k-1} \dpar{E}{\btau}\\
     &=0.
     \end{split}
     \end{equation*}
 \end{proof}

 As for acoustics, these results justify the nice behavior of Active Flux for the Maxwell system. This is related to the
 involutions embedded in the method. Note that 
 the same  counting of Remark \ref{rem:3d}
 also applies to Maxwell, both in two and three space dimensions.

\section{Numerical illustrations.}
The test cases are taken from \cite{VPerrier}. 

\subsection{Linear acoustics}
\subsubsection{Steady vortex}
The first test case is a steady vortex. We take
\begin{subequations}\label{eq:steady}
\begin{equation}
\begin{split}
\bbu(\bbx)&=(-\dpar{\psi}{y}, \dpar{\psi}{x})\\
p(\bbx)&=0
\end{split}
\end{equation} with
\begin{equation}\psi(\bbx)=e^{-0.5\Vert \bby\Vert}, \quad \bby=\frac{\bbx}{0.15}
\end{equation}
\end{subequations} on the domain $[-1,1]^2$. The final time is set to 3.
The temporal scheme is the third order RK-SSP scheme. Clearly, the vorticity evaluated at the vertices evaluated
from \eqref{rot discret} stays invariant in time, up to truncation errors.
\begin{figure}[h]
\begin{center}
   \subfigure[Problem \eqref{eq:steady}]{ \includegraphics[width=0.5\textwidth]{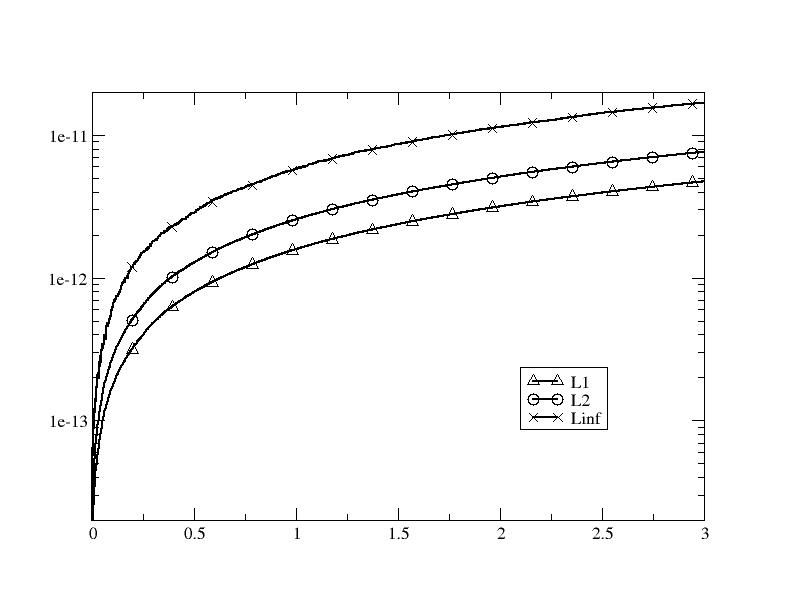}}\subfigure[Problem \eqref{eq:vortex periodique}]{\includegraphics[width=0.5\textwidth]{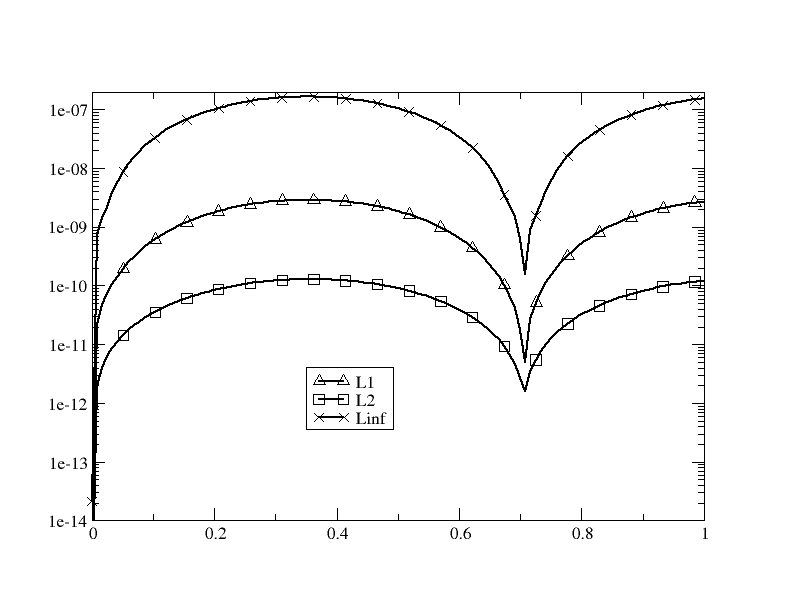}}
\end{center}
\caption{\label{fig:erreur} $L^1$, $L^2$ and $L^\infty$ norm of $\ROT(\bbu-\bbu_0)$ with $\ROT\bbv$ defined by \eqref{rot discret}. This is  evaluated at the vertices.}
\end{figure}

\subsubsection{Unsteady vortex}
The initial solution is the sum of $(p_0,u_0, v_0)^T$ defined as  (with $\bbx=(x,y)$)
\begin{equation*}
\begin{split}
p_0(\bbx)&=\frac{\omega}{c}\sin\big ( k_{\bot}\pi  y\big)\cos\big (k_{\parallel}\pi x)\\
u_0&=\pi k_{\parallel} \cos\big (k_{\bot}\pi y\big )\sin(k_{\parallel}\pi x\big ), \\
v_0&=\pi k_{\bot.      } \sin\big( k_{\bot}\pi y\big )\cos(k_{\parallel}\pi x\big ) 
\end{split}
\end{equation*} and the perturbation:
\begin{equation*}
\begin{split}
\delta p&=0\\
\delta u&=-2 K_0\alpha \bar y\frac{e^{-\alpha/(1-(\bar r)^2)}}{ (1-(\bar r)^2)^2}\\
\delta v&=2 K_0\alpha \bar x\frac{e^{-\alpha/(1-(\bar r)^2)}}{ (1-(\bar r)^2)^2}\\
\end{split}
\end{equation*}
if $\bar r= \Vert \bbx\Vert/r_0<1$ and $(\delta u, \delta u, \delta v)=0$ else.
The domain is $[0,1]^2$, $k_\bot=k_\parallel=2$, $\omega=\pi c\sqrt{k_\bot^2+k_\parallel^2}$, $c=1$, and $K_0=100$, $\alpha=4$, $r_0=0.35$.

We see that $\curl\bbu_0=0$ and $\curl \delta \bbu\neq 0$, so it is a way to see how is preserved the curl, or not. The exact solution is
\begin{equation}\label{eq:vortex periodique}
\big (p(\bbx,t), \bbu(\bbx,t)\big )^T=\big ( p, \bbu_0+\delta \bbu)^T(x, y-\omega t\big ).\end{equation}
We compute the errors and the variation of the discrete curl at time $T_{fin}=1$. For the values of the parameters, the period is $T=1/\sqrt{2}$, and we clearly see this on Figure \ref{fig:erreur}.

\begin{table}[h]
\begin{center}
\begin{tabular}{|c|cc|}
\hline
$h$ & $L^1$ error & slope\\
\hline
   $7.22\; 10^{-2}  $&  $ 6.78\; 10^{-2} $& -\\
   $ 3.61\; 10^{-2} $&  $ 8.83\; 10^{-3} $& $2.94 $ \\
    $1.80\; 10^{-2}  $&  $ 1.122\;10^{-3} $& $2.97 $\\
    $4.51\; 10^{-3}  $&  $ 2.01\; 10^{-5} $& $2.90 $\\
    \hline
\end{tabular}
\end{center}
\caption{Error for the average values, problem \eqref{eq:vortex periodique}.}
\end{table}

We see that we recover the expected order, and the initial curl is well preserved, point-wise, as it can be seen on Figure \ref{fig:erreur}-(b).

\subsection{Maxwell equations}
\subsubsection{Steady problem}
Here, we take, in $[-0.5,0.5]^2$
\begin{equation}\label{maxwell:1}
\begin{split}
E&=0\\
\bbB&=\nabla \psi
\end{split}
\end{equation}
where
\begin{equation}\label{psi}\psi(\bbx)=K_0e^{-\alpha \Vert \bbx-\bbx_c\Vert/r_0}\end{equation}
with again $\bbx=(x,y)$, $\bbx_c=0$, $K_0=1$, $r_0=0.15$ and $\alpha=0.5$.
 
 The $L^1$, $L^2$ and $L^\infty$ norm of $\DIV \bbB$ are shown on figure \ref{fig:div}.
 \begin{figure}[h]
 \begin{center}
 \subfigure[]{ \includegraphics[width=0.45\textwidth]{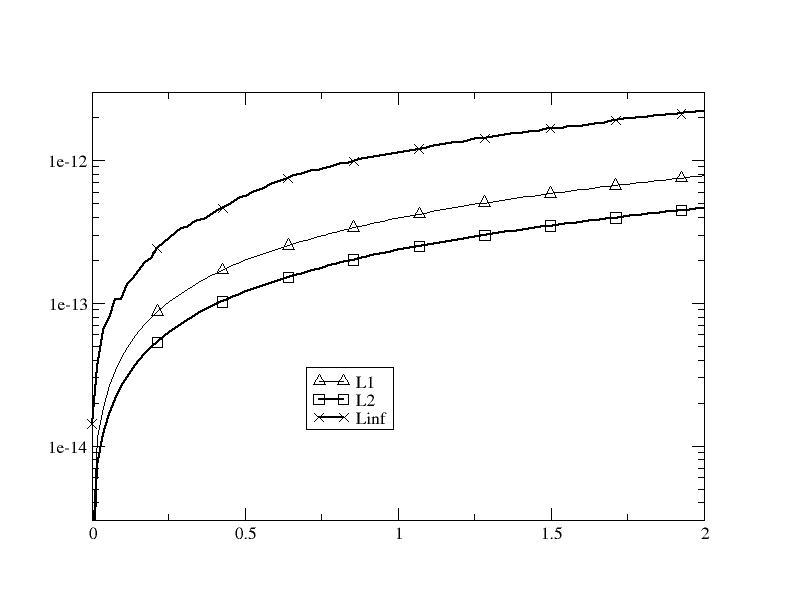}}
\subfigure[]{\includegraphics[width=0.45\textwidth]{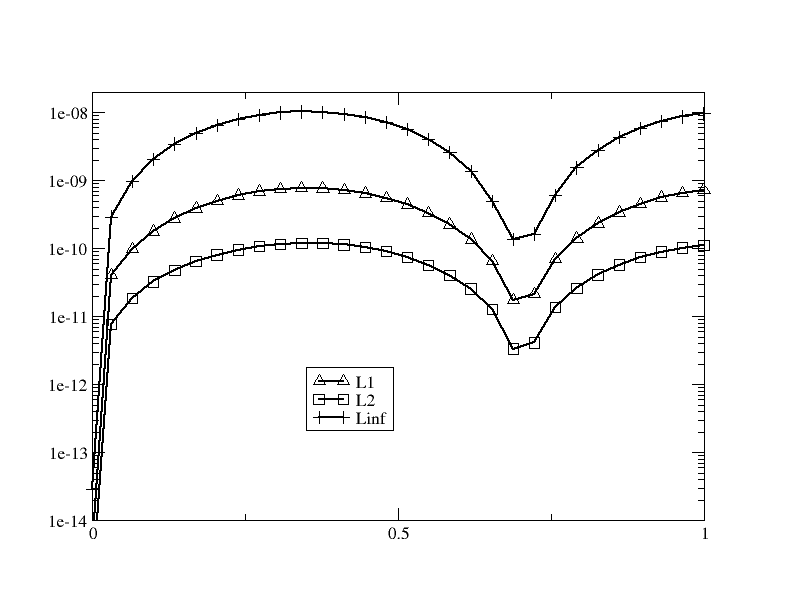}}
\end{center}
\caption{\label{fig:div} (a): Time evolution of $\DIV \bbB$ for problem \eqref{maxwell:1}, (b): Time evolution of $\DIV\big ( \bbB-\bbB_0\big )$ for problem \eqref{maxwell:2}. $\DIV\bbb$ is defined by \eqref{div discret}. This is evaluated at the vertices.
}
\end{figure}
The divergence is very well preserved.

\subsubsection{Unsteady problem}
On $[0,1]^2$ with periodic boundary conditions, we consider the initial conditions
\begin{equation}\label{maxwell:2}
(E_0,\bbB_0)^T=(e, \bbb)^T+(0,\delta \bbb)^T
\end{equation}
where
\begin{equation}
\begin{split}
e&=\omega \cos\big(k_\bot\pi y\big)\sin\big(k_\parallel \pi x\big)\\
B_1&=-k_\bot \pi \sin\big(k_\bot \pi y\big)\cos\big(k_\parallel \pi x\big)\\
B_2&=k_\parallel\pi\cos\big(k_\bot \pi y\big)\sin\big(k_\parallel \pi x\big )
\end{split}
\end{equation}
and
\begin{equation}
\delta \bbb=\left \{\begin{array}{ll}
\nabla \psi & \text{ if} \Vert \bbx\Vert <r_0\\
0& \text{ else}
\end{array}
\right.
\end{equation}
We have used the following parameters and functions:
\begin{itemize}
\item $k_\bot=k_\parallel=2$, $\omega=\pi \sqrt{ k_\bot^2+k_\parallel^2}$,
\item $\psi$ as in \eqref{psi} with $K_0=100$, $r_0=0.35$ and $\bbx_c=(0.5,0.5)^T$
\end{itemize}

The exact solution is given by $(E,\bbB)(x,y,t)=(E_0,\bbB_0)(x,y-\omega t)$. With these parameters, the period is $1/\sqrt{2}$. 
The divergence preservation is tested on a coarse mesh with only $4029$ points, this is shown on figure \ref{fig:div}-(b). The error in the $L^1$ norm is given on Table \ref{table:div}. The error is as expected. 
\begin{table}[h]
\begin{center}
\begin{tabular}{|c|cc|}
\hline
$h$ & $L^1$ error & slope\\
\hline
  $ 3.055\;10^{-2} $&    $2.369\;10^{-2} $&-\\
   $ 1.522\;10^{-2} $ &   $2.932\;10^{-3}  $& $3.01 $ \\
   $ 7.637\;10^{-3}  $&   $3.626\;10^{-4}  $& $3.01 $\\
   $ 3.818\;10^{-3}  $&  $2.278\;10^{-5}   $& $4.02 $\\
    \hline
\end{tabular}
\end{center}

\caption{\label{table:div} Error for the average values, problem \eqref{maxwell:2}.}
\end{table}
\section{Conclusions and perspectives}
We have studied the involution  properties of a semi-discrete version of the Active Flux method
for both linear acoustics and the Maxwell equations.
We have explained why, in the case of a mesh made of triangles in 2D and tetrahedrons in 3D, the semi-discrete Active flux methods preserves very well the structure of curl free or solenoidal solutions (depending on the system).  The analysis provides explicit expressions of the nodal discrete curl and divergence operators
preserved by the scheme, and applies both to the 
 semi-discrete version,  as well as to the fully discrete one, as long as explicit time stepping is used. 

This  work however does not tell the whole story. As remarked along the text,   the analysis  does not apply
as it is to quadrangular and polygonal meshes. Secondly, it does not show whether or not it is true that if the initial condition is the gradient of some scalar function (or curl of some vector potential in the case of the Maxwell equations), then this is also true for the discrete solution  at all times.  The related question
is whether  the   fully discrete method is also a discretization  of the wave equation.

These aspects  will be explored in  future publications.

\section*{Acknowledgements}
 RA acknowledges the support of the Humbold Association via an Humbold Award, as well many stimulating discussions with Professor E. Sonnedr\"ucker, Max Planck Institute f\"ur Plasma Physics, Garching, Germany. MR is a member of the CARDAMOM team at the Inria center at University of Bordeaux.
\appendix
\section{Proof of Lemma \ref{lem:unisolv}\label{app:lemma}}

Let $\alpha_\sigma$, $\sigma\in \mathcal{G}$ and $\alpha_{\bmu'}$, $\bmu'\in \mathcal{I}_k$ be real numbers such that for all $p\in V$
\begin{equation}
\label{test}\sum_{\sigma\in \mathcal{G}}\alpha_{\sigma} \langle\delta_\sigma,p\rangle+\sum_{\bmu'\in \mathcal{I}_k}\alpha_{\bmu'}\langle m_{\bmu'},p\rangle=0.
\end{equation}
We first consider the Lagrange polynomials $L_\sigma$  at  $\sigma$ and define
\begin{equation}
\label{basis function}
\varphi_\sigma=L_\sigma-\sum\limits_{\bmu\in \mathcal{I}_k}
\frac{m_{\bmu}(L_\sigma)}{m_{\bmu}(\psi_{\bmu})}\psi_{\bmu}
\end{equation}
Since $\psi_{\bmu}$ vanishes on $\partial K$ and $L_\sigma$ is a Lagrange polynomial for the Gauss-Lobatto points, we have
$$\langle \delta_{\sigma'},   \varphi_{\sigma}\rangle=\delta_{\sigma}^{\sigma'},$$
and by construction, for all $\bmu\in \mathcal{I}_k$, 
$$\langle m_\bmu, \varphi_{\sigma}\rangle=0.$$
Using this to test \eqref{test}, we see that $\alpha_\sigma=0$ for $\sigma\in \mathcal{G}$.

Then we test with $p=\psi_\bmu$, $\bmu\in \mathcal{I}_k$, that is 
\begin{equation}
\label{basis functions2}\sum_{\bmu'\in \mathcal{I}_k}\alpha_{\bmu'}\int_K  \lambda_1\lambda_2\lambda_3 \blambda^{\bmu'}\blambda^\bmu\; \mathrm d\bbx=0.
\end{equation}
The question reduces to that of the invertibility of the $\vert \II_k\vert \times \vert \II_k\vert$ symmetric matrix
$$A=\begin{pmatrix}\int_K  \lambda_1\lambda_2\lambda_3 \blambda^{\bmu'}\blambda^\bmu\; \mathrm d\bbx
\end{pmatrix}.$$
Taking $X\in \R^{\vert \II_k\vert}$, we see that
$$X^TAX=\int_K \lambda_1\lambda_2\lambda_3 \varphi^2(\bbx)\; \mathrm d\bbx$$
with
$$\varphi=\sum_{\bmu\in \II_k}X_{\bmu} \blambda^\bmu.$$
Since $\{\blambda^\bmu\}_{\bmu\in \II_k}$ is free, the matrix is positive definite. Thus the forms are unisolvant.

For example, for $k=3$ (which is the first non trivial example), {we have, for the moment $ m_\bmu$ with $\bmu=(\mu_1,\mu_2,\mu_3)$ where the $\mu_1+\mu_2+\mu_3=1$, all the $\mu_j$s are all equal to $0$ except for one which is equal to $1$. This leads to 
$$\varphi_\bmu=\lambda_1\lambda_2\lambda_3(a_{1}^k\lambda_1+a_{2}^k\lambda_2+a_3^k\lambda_3).$$} The matrix 
$$A=\frac{1}{6!}\begin{pmatrix}
6 & 4& 4\\
4&6&4\\
4&4&6\end{pmatrix}$$ is circulant. 
{Since} $X{=}(a_1,a_2,a_3)^T=
(\tfrac{1800}{7},-\frac{720}{7},-\frac{720}{7})^T$ is the solution of
$$A X=\begin{pmatrix} 1\\ 0 \\0\end{pmatrix}$$
the polynomials are
\begin{itemize}
    \item Moment $\lambda_1$: $\varphi_{(1,0,0)}=\lambda_1\lambda_2\lambda_3(a_1\lambda_1+a_2\lambda_2+a_3\lambda_3)$.
    \item Moment $\lambda_2$: $\varphi_{(0,1,0)}=\lambda_1\lambda_2\lambda_3(a_3\lambda_1+a_1\lambda_2+a_2\lambda_3)$.
    \item Moment $\lambda_3$: $\varphi_{(0,0,1)}=\lambda_1\lambda_2\lambda_3(a_2\lambda_1+a_3\lambda_2+a_1\lambda_3)$.
    \end{itemize}

\section{Well posedness of the nodal weighting\label{app:1}}

We start with the acoustics system, and  evaluate $\bbK\bbn_\bsigma^K$,
the Jacobian of the system in the direction $\bbn=\bbn^K_\bsigma$: 
$$\bbK^K_\bbn=\begin{pmatrix}
0&0&\bbn_x\\
0&0&\bbn_y\\
\bbn_x&\bbn_y&0
\end{pmatrix}.$$
We write $\bbn=\Vert\bbn\Vert (\cos\theta\sin\theta)$, and the eigenvalues are $0$, $\pm \Vert\bbn\Vert$, the eigenvectors are (for a unit normal)
$$
\bbr_0=\begin{pmatrix}
-\sin\theta\\ \cos \theta \\ 0
\end{pmatrix}, \quad \bbr_+=
\frac{1}{\sqrt{2}}
\begin{pmatrix}
\cos\theta\\
\sin\theta\\
1\end{pmatrix}, \quad \bbr_-=
\frac{1}{\sqrt{2}}
\begin{pmatrix}
-\cos\theta\\
-\sin\theta\\ 
1
\end{pmatrix}.$$ The   factor $1/\sqrt{2}$ ensures
 to have an orthonormal basis of eigenvectors.
Let us   estimate 
$$\sum_{K, \bsigma\in K} (\bbK_{\bbn_\bsigma}^K)^+=\frac{1}{2}\sum_{K, \bsigma\in K} \Vert \bbn^K_\bsigma\Vert (\bbr^K_+\otimes \bbr^K_++\bbr^K_-\otimes \bbr^K_-)=
\frac{1}{2}\bigg ( \sum_{K, \bsigma\in K}\Vert \bbn_\bsigma^K\Vert\bigg )\Id -\frac{1}{2}\sum_{K, \bsigma\in K} \Vert n_\bsigma^K\Vert \bbr^K_0\otimes \bbr^K_0
$$
Since 
$$\bbr^K_0\otimes \bbr^K_0=\begin{pmatrix} \sin^2(\theta_\bsigma^K)& -\sin(\theta_\bsigma^K)\cos(\theta_\bsigma^K) & 0\\
-\sin(\theta_\bsigma^K)\cos(\theta_\bsigma^K) & \cos^2(\theta_\bsigma^K)&0\\
0&0&0\end{pmatrix},$$ we see that 
$$
\sum_{K, \bsigma\in K}\big (\bbK\bbn_\bsigma^K\big )^+
=
\frac{1}{2}\sum_{K, \bsigma\in K}\Vert \bbn_\bsigma^K\Vert
\begin{pmatrix}
     \cos^2(\theta_\bsigma^K) & \sin(\theta_{\bsigma^K})\cos(\theta_{\bsigma^K}) & 0\\
     \sin(\theta_\bsigma^K)\cos(\theta_\bsigma^K)& \sin^2(\theta_\bsigma^K) &0\\
    0&0&1
\end{pmatrix}
$$
Evaluating  explicitly $\bbx^K\bigg ( \sum_{K, \bsigma\in K}\big (\bbK\bbn_\bsigma^K\big )^+\bigg )\bbx$ for $\bbx=(x,y,z)^T$, we  obtain after few manipulations
$$
\begin{aligned}
\bbx^K\bigg ( \sum_{K, \bsigma\in K}\big (\bbK\bbn_\bsigma^K\big )^+\bigg )\bbx 
= &
\sum_{K, \bsigma\in K}
\Vert \bbn_\sigma^K\Vert \left\{
 \bigg ( \cos^2(\theta_\bsigma^K) x^2+2\sin(\theta_\bsigma^K)\cos(\theta_\bsigma^K)xy+\sin^2(\theta_\bsigma^K)y^2\bigg )+  z^2
\right\}\\=&
\sum_{K, \bsigma\in K}
\left\{ 
 \dfrac{1}{\|\mathbf{n}_{\bsigma}^K\|}((x,y)\cdot\mathbf{n}_{\bsigma}^K )^2 +\Vert \bbn_\sigma^K\Vert  z^2
\right\}.
\end{aligned}
$$
 The sum can be zero only if $z=0$, and if 
  $\mathbf{x}\cdot\mathbf{n}_{\bsigma}^K =0$ for all $K$.
    For all internal nodes $\sum_K\bbn_\bsigma^K=0$, so this is never possible   unless $\bbx=0$. So this means the matrix is positive definite, and symmetric, so always invertible.
    For boundary points the matrix bay be only semi-definite,  thus not necessarily invertible, although  only   in exceptional  geometrical circumstances. For that reason, we set $\varepsilon=0$ everywhere except for the boundary points.\\

    For the Maxwell equations we can proceed similarly. First we write the Jacobian
$$\bbK^K_\bsigma=\begin{pmatrix}
0&0&\bbn_y\\
0&0&-\bbn_x\\
\bbn_y&-\bbn_x&0
\end{pmatrix}.$$
We introduce the tangent vectors $\mathbf{t}_{\bsigma}^K=(t_x,t_y ):=( n_y,\, -n_x )$
which trivially satisfy
\begin{equation}\label{eq:tt}
\| \mathbf{t}_{\bsigma}^K\| = 
\| \mathbf{n}_{\bsigma}^K\|  \;,\quad\sum\limits_{K,\bsigma\in K}\mathbf{t}_{\bsigma}^K =0.
\end{equation}
We can  write now
$$
\bbK^K_\bsigma=\begin{pmatrix}
0&0&\mathbf{t}_x\\
0&0&\mathbf{t}_y\\
\mathbf{t}_x&\mathbf{t}_y&0
\end{pmatrix}.$$
 In particular, we find the same expressions  for eigenvectors and egivenvalues, replacing the normal by the tangent. We can also set
  $\mathbf{t}=\| \mathbf{n}\| (\cos\beta, \sin\beta )$, where
   $\beta=\pi/2 -\theta$.
In particular, the eigenvalues are still  $0$, $\pm \Vert\bbn\Vert$,
 while the eigenvectors read (for a unit normal/tangent)
$$
\bbr_0=\begin{pmatrix}
-\sin\beta\\ \cos \beta \\ 0
\end{pmatrix}, \quad \bbr_+=
\frac{1}{\sqrt{2}}
\begin{pmatrix}
\cos\beta\\
\sin\beta\\
1\end{pmatrix}, \quad \bbr_-=
\frac{1}{\sqrt{2}}
\begin{pmatrix}
-\cos\beta\\
-\sin\beta\\ 
1
\end{pmatrix}.$$
Proceeding   exactly as before one  shows that
$$
\sum_{K, \bsigma\in K}\big (\bbK\bbn_\bsigma^K\big )^+
=
\frac{1}{2}\sum_{K, \bsigma\in K}\Vert \bbn_\bsigma^K\Vert
\begin{pmatrix}
     \cos^2(\beta_\bsigma^K) & \sin(\beta_{\bsigma^K})\cos(\beta_{\bsigma^K}) & 0\\
     \sin(\beta_\bsigma^K)\cos(\beta_\bsigma^K)& \sin^2(\beta_\bsigma^K) &0\\
    0&0&1
\end{pmatrix}
$$
and that  
$$ 
\bbx^K\bigg ( \sum_{K, \bsigma\in K}\big (\bbK\bbn_\bsigma^K\big )^+\bigg )\bbx 
=  
\sum_{K, \bsigma\in K}
\left\{ 
 \dfrac{1}{\|\mathbf{n}_{\bsigma}^K\|}((x,y)\cdot\mathbf{t}_{\bsigma}^K )^2 +\Vert \bbn_\sigma^K\Vert  z^2
\right\}.
$$
As before, this quantity can only vanish if simultaneously we have $z=0$,
and $\mathbf{x}\cdot\mathbf{t}_{\bsigma}^K=0$  in all elements.
This is not possible due to the second in \eqref{eq:tt}. As before,
this allows to set $\varepsilon=0$ everywhere except for  boundary points.

\printbibliography
\end{document}

%% file: command.tex
\usepackage{fullpage}

\usepackage{amsmath,amsfonts,amsthm,amssymb}
\usepackage{mathrsfs}
\usepackage{color,xcolor}
\usepackage{ifpdf}
\usepackage{psfrag}
\usepackage{graphicx,graphics,subfigure,epsfig}
\usepackage{url}
\usepackage{hyperref}
\usepackage{xspace}
\usepackage{bm}
\usepackage{dutchcal}

\usepackage[colorinlistoftodos,prependcaption,textsize=small]{todonotes}
\newtheorem{theorem}{Theorem}[section]

\newtheorem{lemma}[theorem]{Lemma}

\newtheorem{proposition}[theorem]{Proposition}
\newtheorem{remark}[theorem]{Remark}

\newcommand{\dpar}[2]{\dfrac{\partial #1}{\partial #2}}

\renewcommand{\div}{\text{ div }}

 \newcommand{\R}{\mathbb R}

\renewcommand{\P}{\mathbb P}

\newcommand{\II}{\mathcal I}

\newcommand{\bba}{\mathbf{a}}
\newcommand{\bbb}{\mathbf{b}}

\newcommand{\bbf}{\mathbf{f}}

\newcommand{\bbn}{\mathbf{n}}

\newcommand{\bbr}{\mathbf{r}}

\newcommand{\bbu}{\mathbf{u}}
\newcommand{\bbv}{\mathbf{v}}

\newcommand{\bbx}{\mathbf{x}}
\newcommand{\bby}{\mathbf{y}}

\newcommand{\bbA}{\mathbf{A}}
\newcommand{\bbB}{\mathbf{B}}

\newcommand{\bbK}{\mathbf{K}}

\newcommand{\bbN}{\mathbf{N}}

\newcommand{\btau}{{\bm\tau}}
\newcommand{\bsigma}{{\bm\sigma}}
\newcommand{\bomega}{{\bm\omega}}

\newcommand{\bmu}{{\bm {\mu}}}
\newcommand{\blambda}{{\bm {\lambda}}}

\newcommand{\dof}{\text{DoFs}}

\newcommand{\Id}{\text{Id }}

\newcommand{\curl}{{ \mathbf{ curl }\;}}
\newcommand{\grad}{{ \mathbf{ grad }\;}}

\newcommand{\rot}{\text{ rot }}
\renewcommand{\div}{\text{ div }}

\newcommand{\DIV}{\textbf{\text{DIV }}}
\newcommand{\ROT}{\textbf{\text{ROT }}}

%% file: GL_tri.tikz
{\pgfkeys{/pgf/fpu/.try=false}%
 \ifdim\dimen1=0pt\ifdim\dimen3=0pt\dimen1=1000sp\dimen3\dimen1
  \else\dimen1\dimen3\fi\else\ifdim\dimen3=0pt\dimen3\dimen1\fi\fi
\begin{tikzpicture}[x=+\dimen1, y=+\dimen3]
{\ifx\XFigu\undefined\catcode`\@11
\def\temp{\alloc@1\dimen\dimendef\insc@unt}\temp\XFigu\catcode`\@12\fi}
\XFigu3946sp
\ifdim\XFigu<0pt\XFigu-\XFigu\fi
\tikzset{inner sep=+0pt, outer sep=+0pt}
\pgfsetfillcolor{black}
\pgftext[base,left,at=\pgfqpointxy{9825}{-8100}] {\fontsize{12}{14.4}\usefont{T1}{ptm}{m}{n}$\xi_3$}
\pgfsetlinewidth{+7.5\XFigu}
\pgfsetcolor{black}
\filldraw  (9750,-7725) circle [radius=+75];
\filldraw  (5475,-2100) circle [radius=+75];
\filldraw  (2175,-6975) circle [radius=+75];
\filldraw  (2775,-8025) circle [radius=+75];
\filldraw  (4725,-3225) circle [radius=+75];
\filldraw  (6300,-3150) circle [radius=+75];
\filldraw  (8925,-6600) circle [radius=+75];
\filldraw  (8400,-7800) circle [radius=+75];
\draw (1425,-8100)--(5475,-2100)--(9825,-7725)--(1425,-8100);
\pgfsetdash{{+60\XFigu}{+60\XFigu}}{++0pt}
\draw (6300,-3150)--(2775,-8025);
\draw (4725,-3225)--(8400,-7800);
\draw (2175,-6975)--(8925,-6600);
\draw (2175,-6975)--(2775,-8025);
\draw (8925,-6600)--(8400,-7800);
\draw (4725,-3225)--(6300,-3150);
\pgftext[base,left,at=\pgfqpointxy{1350}{-8550}] {\fontsize{12}{14.4}\usefont{T1}{ptm}{m}{n}$\xi_0$}
\pgftext[base,left,at=\pgfqpointxy{2775}{-8475}] {\fontsize{12}{14.4}\usefont{T1}{ptm}{m}{n}$\xi_1$}
\pgftext[base,left,at=\pgfqpointxy{8400}{-8325}] {\fontsize{12}{14.4}\usefont{T1}{ptm}{m}{n}$\xi_2$}
\pgfsetdash{}{+0pt}
\filldraw  (1425,-8100) circle [radius=+75];
\end{tikzpicture} }%